\documentclass[reqno,11pt,final]{amsart}
\usepackage{glaudo_pkg_en}

\title[Finer estimates on the $2$-d matching problem]
{Finer estimates on the $2$-dimensional matching problem}
\author{L. Ambrosio}
\address{Scuola Normale Superiore, Piazza dei Cavalieri 7, 56126 Pisa, Italy}
\email{luigi.ambrosio@sns.it}
\author{F. Glaudo}
\address{ETH, Rämistrasse 101, 8092 Zürich, Switzerland}
\email{federico.glaudo@math.ethz.ch}
\begin{document}
\begin{abstract}
  We study the asymptotic behaviour of the expected cost of the random matching problem
  on a $2$-dimensional compact manifold, improving in several aspects the results of \cite{Ambrosio-Stra-Trevisan2018}.
  In particular, we simplify the original proof (by treating at the same time upper and lower bounds) and
  we obtain the coefficient of the leading term of the asymptotic expansion of the expected cost for the
  random bipartite matching on a general $2$-dimensional closed manifold. We also sharpen the estimate of 
  the error term given in \cite{Ledoux18} for the semi-discrete matching.
  
  As a technical tool, we develop a refined contractivity estimate for the heat flow 
  on random data that might be of independent interest.
\end{abstract}

\maketitle
\tableofcontents
\thispagestyle{empty} %Removes the ugly page number on the first page.

\section{Introduction}
The bipartite matching problem is a very classical problem in computer science. It asks to find, among all
possible matching in a bipartite weighted graph, the one that minimizes the sum of the costs of the chosen
edges. The most typical instance of the matching problem arises when trying to match two families 
$(X_i)_{1\le i\le n}$ and $(Y_i)_{1\le i\le n}$  of $n$ points in a metric space $(M, d)$ and the cost of
matching two points $X_i$ and $Y_j$ depends only on the distance $d(X_i, Y_j)$.

We will investigate a random version of the problem, that is, in its most general form, the following one.
\begin{problem*}[Random Bipartite Matching]
  Let $(M, d, \m)$ be a metric measure space such that $\m$ is a probability measure and let $p\ge 1$ be
  a fixed exponent. 
  Given two families $(X_i)_{1\le i\le n}$ and $(Y_i)_{1\le i\le n}$ of independent random points
  $\m$-uniformly distributed on $M$, study the value of the expected matching cost
  \begin{equation*}
    \E{\min_{\sigma\in S^n} \frac1n\sum_{i=1}^n d^p\left(X_i, Y_{\sigma(i)}\right)} \fullstop
  \end{equation*}
  For brevity we will denote with $\mathbb E_n$ the mentioned expected value, dropping the dependence on
  $(M, d, \m)$ and $p>0$.
\end{problem*}
The coefficient $\frac1n$ before the summation is inserted both for historical reasons and because it will
make things easier when we will move to the context of probability measures.

Let us remark that, without any further assumption on $M$ and $\m$, the statement of the problem is too general to be interesting.

In the special case $M=\cc01^d$, $\m=\restricts{\Leb^d}{M}$ the problem has been studied deeply in 
the literature. 
As a general reference, we suggest the reading of the book \cite{Talagrand14}, which devotes multiple chapters to 
the treatment of the random matching problem.
Before summarizing the main known results in this setting, let us remark that also the weighted setting ($M=\R^d$, $\m$ a generic measure with adequate moment estimates) has attracted a lot of attention since it is the most useful in applications (see \cite{DereichScheutzowSchottstedt2013,BoissarLeGouic2014,FournierGuillin2015,WeedBach2017}).

\subsubsection*{Dimension $=1$} 
  When $d=1$, so $M=\cc01$, the problem is much easier compared to
  other dimensions. Indeed on the interval a monotone matching is always optimal.
  Thus the study of $\mathbb E_n$ reduces to the study of the probability distribution of 
  the $k$-th point in the increasing order (that is $X_k$ if the sequence $(X_i)$ is assumed to be 
  increasing). In particular it is not hard to show
  \begin{equation*}
    \mathbb E_n \approx n^{-\frac p2} \fullstop
  \end{equation*}
  In the special case $d=1$ and $p=2$ we can even compute $\mathbb E_n$ explicitly
  \begin{equation*}
      \mathbb E_n = \frac 1{3(n+1)} \fullstop
  \end{equation*}
  A monograph on the $1$-dimensional case where the mentioned results, and much more, can be found 
  is \cite{Bobkov2014}.
\subsubsection*{Dimension $\ge 3$} 
  When $d\ge 3$, for any $1\le p < \infty$ it holds
  \begin{equation*}
    \mathbb E_n \approx n^{-\frac pd} \fullstop
  \end{equation*}
  For $p=1$, the result is proven in \cite{Dobric1995,Talagrand1992}, whereas the
  paper \cite{Ledoux2017} addresses all cases $1\le p < \infty$ with methods, 
  inspired by \cite{Ambrosio-Stra-Trevisan2018}, similar to the ones we are going to use.
  
  In \cite[Theorem 2]{Barthe2013} the authors manage to prove the existence of the limit of the renormalized cost
  \begin{equation*}
    \lim_{n\to\infty}\mathbb E_n \cdot n^{\frac pd}
  \end{equation*}
  under the constraint $1\le p < d/2$, but the value of the limit is not determined.
\subsubsection*{Dimension $=2$} 
  When $d=2$, the study of $\mathbb E_n$ becomes suddenly more delicate. 
  As shown in the fundamental paper \cite{Ajtai1984}, for any $1\le p <\infty$, the growth is
  \begin{equation*}
    \mathbb E_n \approx \left(\frac{\log(n)}{n}\right)^{p/2} \fullstop
  \end{equation*}
  Their proof is essentially combinatorial and following such a strategy there is little hope to 
  be able to compute the limit of the renormalized quantity
  \begin{equation*}
    \lim_{n\to\infty} \mathbb E_n\cdot \left(\frac{n}{\log(n)}\right)^{p/2} \fullstop
  \end{equation*}
  Much more recently, in 2014, in \cite{CaraccioloEtAl2014} the authors claimed that, if $p=2$, the 
  limit value is $\frac1{2\pi}$ with an ansatz supporting their claim (see also \cite{Caracciolo2014one} for 
  a deeper analysis of the 1-dimensional case).
  Then in \cite{Ambrosio-Stra-Trevisan2018} it was finally proven that the claim is indeed true.
  The techniques used in this latter work are completely different from the combinatorial approaches
  seen in previous works on the matching problem, indeed the tools used come mainly from the theory
  of partial differential equations and optimal transport.
  
\subsubsection*{Semi-discrete matching problem, large scale behaviour} 
In the semi-discrete matching problem, a single family of independent and identically distributed points
$(X_i)_{1\le i\le n}$ has to be matched to the reference measure $\m$. 
By rescaling, and possibly replacing the empirical measures with
a Poisson point process, the semi-discrete matching problem can be connected to
the Lebesgue-to-Poisson transport problem of \cite{HuesmannSturm2013}, 
see \cite{Goldman2018} where the large scale behaviour of the optimal maps
is deeply analyzed.    

As in \cite{Ambrosio-Stra-Trevisan2018}, we will focus on the case where $(M, d)$ is a $2$-dimensional 
compact Riemannian manifold, $\m$ is the volume measure and the cost is given by the square of the distance.
From now on we are going to switch from the language of combinatorics and computer science to the 
language of probability and optimal transport. Thus, instead of matching two family of points we will 
minimize the Wasserstein distance between the corresponding empirical measures.

We will prove the following generalization of \cite[Eq. (1.2)]{Ambrosio-Stra-Trevisan2018} 
to manifold different from the torus and the square.
\begin{theorem}[Main Theorem for bipartite matching]\label{thm:main_theorem_bipartite}
  Let $(M,\metric)$ be a $2$-dimensional compact closed manifold (or the square $\cc01^2$) whose 
  volume measure $\m$ is a probability.
  Let $(X_i)_{i\in\N}$ and $(Y_i)_{i\in\N}$ be two families of independent random points 
  $\m$-uniformly distributed on $M$. Then
  \begin{equation*}
    \lim_{n\to\infty} \frac{n}{\log(n)}\cdot
    \E{W^2_2\left(\frac1n\sum_{i=1}^n\delta_{X_i}, \frac1n\sum_{i=1}^n\delta_{Y _i}\right)}
    =\frac1{2\pi} \fullstop
  \end{equation*}
\end{theorem}

In the context of the semi-matching problem we simplify 
the proof contained in \cite{Ambrosio-Stra-Trevisan2018} and strengthen the estimate of the error 
term provided in \cite{Ledoux18}.

\begin{theorem}[Main Theorem for semi-discrete matching]\label{thm:main_theorem_semidiscrete}
  Let $(M,\metric)$ be a $2$-dimensional compact closed manifold (or the square $\cc01^2$) 
  whose volume measure $\m$ is a probability. 
  Let $(X_i)_{i\in\N}$ be a family of independent random points $\m$-uniformly distributed on $M$.
  There exists a constant $C=C(M)$, such that
  \begin{equation*}
    \abs*{\E{W^2_2\left(\frac1n\sum_{i=1}^n\delta_{X_i}, \m\right)} - \frac{\log(n)}{4\pi n}} 
    \le C \frac{\sqrt{\log(n)\log\log(n)}}n \fullstop
  \end{equation*}
\end{theorem}

In order to describe our approach let us focus on the semi-discrete matching problem.
Very roughly, we compute a first-order approximation of the optimal transport from $\m$ to 
$\frac1n\sum \delta_{X_i}$ and we show, overcoming multiple technical difficulties, that \emph{very often}
the said transport is \emph{almost optimal}. 
In some sense, this strategy is even closer to the
heuristics behind the ansatz proposed in \cite{CaraccioloEtAl2014}, compared to the strategy pursued in 
\cite{Ambrosio-Stra-Trevisan2018} (even though many technical points will be in common).

More in detail, here is a schematic description of the proof.

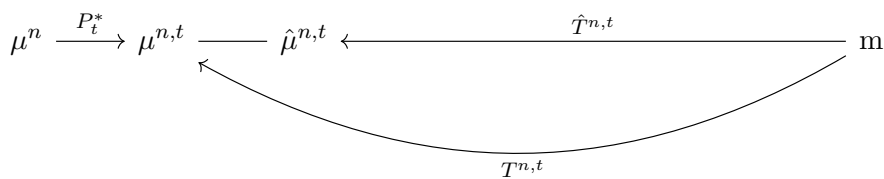
\begin{figure}[htb]
  \centering
  \begin{tikzcd}
\mu^n \arrow[r,"P_t^*"] &
\mu^{n,t} \arrow[r, dash] &
\hat\mu^{n,t} &[15em]
\m \arrow[l, "\hat T^{n,t}"'] \arrow[ll, bend left, "T^{n,t}"]
\end{tikzcd}
  \caption{Sketch of the proof strategy.}
\end{figure}

\begin{enumerate}
  \item Let us denote $\mu^n\defeq \frac1n\sum \delta_{X_i}$ the empirical measure. We construct a regularized
    version of $\mu^n$, called $\mu^{n,t}$, that is \emph{extremely} near to $\mu^n$ in the Wasserstein 
    distance.
  \item We consider a probabilistic event $A^{n,t}_\xi$ similar to $\{\norm{\mu^{n,t}-\m}_\infty<\xi\}$ and
    show that such an event is \emph{extremely} likely. The event considered is rather technical, 
    but should be understood as the intuitive event: ``the $X_i$ are well-spread on $M$''.
  \item With a known trick in optimal transport (Dacorogna-Moser coupling), we construct a transport map $T^{n,t}$ 
    from $\m$ to $\mu^{n,t}$.
  \item In the event $A^{n,t}_\xi$, we derive from $T^{n,t}$ an \textbf{optimal} map 
    $\hat T^{n,t}$ from $\m$ into a measure $\hat \mu^{n,t}$ that is extremely near 
    in the Wasserstein distance to $\mu^{n,t}$.
  \item We conclude computing the average cost of $\hat T^{n,t}$.
\end{enumerate}
The regularized measure $\mu^{n,t}$ is obtained from $\mu$ through the heat flow, namely 
$\mu^{n,t}=P^*_t(\mu^n)$ where $t>0$ is a suitably chosen time (see \cref{subsec:heat_flow_reg} for the definition of $P^*_t$). In 
\cref{sec:extreme_contractivity} we develop an improved contractivity estimate for the heat 
flow on random data and we use it to show that $\mu^n$ and $\mu^{n,t}$ are sufficiently near 
in the Wasserstein distance.
The probabilistic event $A^{n,t}_\xi$ is defined and studied in \cref{sec:probability}. Let us remark
that this is the only section where probability theory plays a central role.
The map $T^{n,t}$ is constructed in \cref{sec:transport_ineq} as the flow of a vector field. 
The map $\hat T^{n,t}$ is simply the exponential map applied on (a slightly modified version of) 
the same vector field. Optimality of $\hat T^{n,t}$ follows from the fact that it 
has small $C^2$-norm (see \cite[Theorem 1.1]{Glaudo19} or \cite[Theorem 13.5]{Villani09}). We will devote 
\cref{app:stability_flow} to showing that $\mu^{n,t}$ and $\hat\mu^{n,t}$ are sufficiently near in the
Wasserstein distance.

In \cref{sec:semidiscrete_matching,sec:bipartite_matching} we will prove our main theorems for the
semi-discrete matching problem and the bipartite matching problem. The proofs are almost equal.

Differently from the proof contained in \cite{Ambrosio-Stra-Trevisan2018}, we do not need the duality theory
of optimal transport as it is completely encoded in the mentioned theorem stating that a \emph{small} 
map is optimal.
In this way, we do not need to manage the upper-bound and the lower-bound of the expected cost separately.
Last but not least, as will be clear by the comments scattered through the paper, our proof is just one little step away from generalizing the results
also to weighted manifolds and manifolds with boundary.

\subsection*{Acknowledgement}
We thank F.Stra and D.Trevisan for useful comments, during the development of the paper, and the paper's reviewers
for their constructive and detailed observations.
The second author is funded by the European Research Council under the Grant Agreement
No. 721675 “Regularity and Stability in Partial Differential Equations (RSPDE)”.

\section{General Setting and Notation}
The setting and definitions we are going to describe in this section will be used in the whole paper.
Only in the section devoted to the bipartite matching we are going to change slightly some definitions.

Whenever we say $A\lesssim B$ we mean that there exists a positive constant $C=C(M)$ that 
depends only on the manifold $M$ such that $A\le CB$. If we add a subscript like $\lesssim_p$ it 
means that the implicit constant is allowed to depend also on $p$.
Even though this notation might be a little confusing initially, 
it is extremely handy to avoid the introduction of a huge number of meaningless constant factors.

With $\prob(M)$ we will denote the set of Borel probability measures on the manifold $M$.

\subsection{Ambient Manifold}\label{sub:setting}
Let $M$ be a closed compact $2$-dimensional Riemannian manifold and let $\m$ be its volume measure. We
will always work, under no real loss of generality, under the assumption that $\m$ is a probability measure. 
Unless stated otherwise, this will be the ambient space for all our results.

It is very tempting to work in the more general setting of weighted/with boundary manifolds. 
Indeed it might seem that most of what we obtain could be easily achieved also for 
weighted/with boundary manifolds.
Nonetheless, there is an issue that we could not solve. The estimate on the derivatives of the heat kernel
we use, specifically \cref{thm:heat_deriv}, seems to be known in literature only for a 
closed and nonweighted manifold. Apart from that single result (that most likely holds also 
in the weighted setting if the weight is sufficiently smooth) everything else can be easily 
adapted to the weighted and nonclosed setting. 
In an appendix to this work, we manage to extend the mentioned estimate to the case of the 
square (and thus all our results apply also when $M=\cc01^2$). 

By weighted manifold we mean a Riemannian manifold where the measure is perturbed as $\m_V=e^{-V}\m$ where
$V:M\to\R$ is a smooth function (i.e. the weight). The matching problem is very susceptible to the
change of the reference measure (as the case of Gaussian measures, recently considered in
\cite{Ledoux2017,Ledoux18,Talagrand2018}, illustrates) and therefore gaining the possibility to add a weight would broaden
the scope of our results.

Even though we are not able to generalize \cref{thm:heat_deriv}, during the paper we will outline 
what are the changes necessary to make everything else work in the weighted/with boundary case. 

Let us say now the fundamental observation that is needed to handle the weighted/with boundary case: 
we have to adopt the \emph{right} definition of Laplacian.

In the weighted (or even with boundary) setting, the \emph{standard} Laplacian must be replaced by the so called drift-Laplacian (still denoted $\lapl$ for consistency), also named Witten Laplacian, characterized by the identity
\begin{equation}\label{eq:def_laplacian}
  \int_M -\lapl u \cdot\varphi\de \m_V = \int_M \nabla u\cdot\nabla \varphi\de\m_V 
\end{equation}
 for any 
$\varphi\in C^\infty(M)$. This operator is related to the standard Laplace-Beltrami operator $\tilde\lapl$ 
by  $\lapl = \tilde\lapl - \nabla V\cdot\nabla$ (see~\cite{Grigor06}) and, in the case of manifolds with 
boundary, \cref{eq:def_laplacian} encodes the null Neumann boundary condition. 
Using this definition everywhere, almost all the statements and proofs that we provide in the nonweighted closed setting can be adapted straight-forwardly to the weighted/with boundary setting.

\subsection{Random Matching Problem Notation}
\subsubsection{Empirical measures}
Let $(X_i)_{1\le i\le n}$ be a family of independent random points $\m$-uniformly distributed on $M$.
Let us define the empirical measure associated to the family of random points
\begin{equation*}
  \mu^n \defeq \frac1n \sum_{i=1}^n \delta_{X_i} \fullstop
\end{equation*}
When two independent families $(X_i)$ and $(Y_i)$ of random points $\m$-uniformly distributed on $M$ will
be considered, we will denote with $\mu_0^n$ and $\mu_1^n$ the empirical measures associated respectively to
$(X_i)$ and $(Y_i)$.

The main topic of this paper is the study the two quantities
\begin{equation*}
  \E{W_2^2(\mu^n, \m)} \text{ and } \E{W_2^2(\mu_0^n, \mu_1^n)} \fullstop
\end{equation*}

\subsubsection{Wasserstein distance}
The quadratic Wasserstein distance, denoted by $W_2(\emptyparam, \emptyparam)$, is the distance induced on
probability measures by the quadratic optimal transport cost
\begin{equation*}
  W_2^2(\mu, \nu) \defeq \min_{\pi\in\Gamma(\mu, \nu)} \int_{M\times M} d(x, y)^2\de\pi(x, y) \comma
\end{equation*}
where $\Gamma(\mu,\nu)$ is the set of Borel probability measures on the product $M\times M$ whose first and 
second marginals are $\mu$ and $\nu$ respectively. See the monographs \cite{Villani09} or \cite{Santambrogio15} for
further details.
Let us recall that when both measures are given by a sum of Dirac masses the Wasserstein distance becomes the
more elementary bipartite matching cost
\begin{equation*}
  W_2^2(\mu_0^n, \mu_1^n) = \frac1n \min_{\sigma\in S^n}\sum_{i=1}^n d\left(X_i, Y_\sigma(i)\right)^2 \fullstop  
\end{equation*}

\subsubsection{Heat flow regularization}\label{subsec:heat_flow_reg}
For any positive time $t>0$, let $\mu^{n,t}$ be the evolution through the heat flow of $\mu^n$, that is
\begin{equation*}
  \mu^{n,t} = P^*_t(\mu^n) = \left(\frac1n \sum_{i=1}^n p_t(X_i, \emptyparam)\right)\m = u^{n,t}\m \comma
\end{equation*}
where $u^{n,t}$ is implicitly defined as the density of $\mu^{n,t}$ with respect to $\m$. Let us recall that
$P_t^*$ denotes the heat semigroup on the space of measures and $p_t(\emptyparam,\emptyparam)$ is the
heat kernel at time $t$. For some background on the heat flow on a compact Riemannian manifold, 
see for instance~\cite[Chapter 6]{Chavel84}.

Why are we regularizing $\mu^n$ through the heat flow? 
First of all let us address a simpler question: why are we
regularizing at all? The intuition is that regularization allows us to ignore completely the 
\emph{small-scale bad behaviour} that is naturally associated with the empirical measure. 
For example, the regularization is necessary to gain the uniform estimate we will show 
in \cref{thm:prob_flat_density}.

But why the heat flow? A priori any kind of good enough convolution kernel that depends on a parameter 
$t>0$ would fit our needs. Once again the intuition is pretty clear: the heat flow is the \emph{best way} to 
go from the empirical measure to the standard measure. Indeed it is well known from \cite{Jordan-Kinderlehrer-Otto98} that the heat flow can be
seen as the gradient flow in the Wasserstein space induced by the relative entropy functional (see \cite{Erbar2010}
for the extension of \cite{Jordan-Kinderlehrer-Otto98} to Riemannian manifolds). More practically, the semigroup property 
of the heat kernel provides a lot of identities
and estimates and plays a crucial role also in the proof of the
refined contractivity property of \cref{sec:extreme_contractivity}.

\subsubsection{The potential $f^{n,t}$}
It is now time to give the most important definition. Let $f^{n,t}:M\to\R$ be the unique function with null
mean such that
\begin{equation}\label{eq:definition_f}
  -\lapl f^{n,t} = u^{n,t} - 1 \fullstop
\end{equation}
The hidden idea behind this definition, underlying the ansatz of \cite{CaraccioloEtAl2014}, 
is a linearization of the Monge-Ampère equation under the assumption
that $u^{n,t}$ is already extremely near to $1$. We suggest the reading of the introduction of 
\cite{Ambrosio-Stra-Trevisan2018} for a deeper explanation. 
The mentioned linearization hints us that $\nabla f^{n,t}$ should be an approximate optimal map from
the measure $\mu$ to the measure $\mu^{n,t}$. We will see in \cref{prop:transport_bound} and later
in \cref{thm:main_theorem_semidiscrete} that this is indeed true.

Let us remark that, as much as possible, we try to be consistent with the notation used in 
\cite{Ambrosio-Stra-Trevisan2018}.

\section{Flatness of the Regularized Density}\label{sec:probability}
\begin{definition}[Norm of a tensor]\label{def:tensor_infinity_norm}
  Given a $2$-tensor field $\tau\in T^0_2(M)$, let the operator norm at a point $x\in M$ be defined as
  \begin{equation*}
    \abs{\tau(x)} = \sup_{u,v\in T_xM\setminus\{0\}} \frac{\abs{\tau(x)[u,v]}}{\abs{u}\abs{v}} \fullstop
  \end{equation*}
  The infinity norm of the tensor $\tau$ is then defined as the supremum of the pointwise norm
  \begin{equation*}
    \norm{\tau}_\infty = \sup_{x\in M}\abs{\tau(x)}\fullstop
  \end{equation*}
\end{definition}
\begin{remark}
  With this norm, it holds
  \begin{equation*}
    \abs{X(\gamma(1))-\Xpar} \le \Length(\gamma)\cdot\norm{\nabla X}_\infty
  \end{equation*}
  whenever $\gamma:\cc01\to M$ is a smooth curve, $X$ is a smooth vector field and 
  $\Xpar\in T_{\gamma(1)}M$ is the parallel transport of $X(\gamma(0))$ onto $T_{\gamma(1)}M$
  along $\gamma$.
\end{remark}

For a fixed $\xi>0$, we want to investigate how unlikely is the event
\begin{equation*}
  A^{n,t}_\xi \defeq \left\{\norm{\nabla^2 f^{n,t}}_\infty < \xi\right\} \fullstop
\end{equation*}
Let us recall that \cref{eq:definition_f} implies 
\begin{equation*}
  \norm{u^{n,t}-1}_\infty = \norm{-\lapl f^{n,t}}_\infty\lesssim \norm{\nabla^2 f^{n,t}}_\infty \comma
\end{equation*}
thus the event $A^{n,t}_\xi$ is very similar to (and is contained in) the event considered 
in \cite[Proposition 3.10]{Ambrosio-Stra-Trevisan2018}.
Let us remark that taking a larger $t>0$ will of course assure us that $A^{n,t}_\xi$ is extremely likely, but,
as we will see, taking a $t>0$ that is too large is unfeasible.

Our goal is showing the following estimate on the probability of $A^{n,t}_\xi$.
\begin{theorem}\label{thm:prob_flat_density}
  There exists a constant $a=a(M)>1$ such that, for any $n\in\N$, $0<\xi<1$ and $0<t<1$ it holds
  \begin{equation*}
    \P{(A^{n,t}_\xi)^{\complement}}\lesssim\frac1{\xi^2t^3}a^{-nt\xi^2} \fullstop
  \end{equation*}
\end{theorem}
The proof of \cref{thm:prob_flat_density} will follow rather easily via a standard 
concentration inequality once we have established some nontrivial inequalities 
concerning the heat kernel.

In fact, a vast part of this section will be devoted to a fine study of $q_t$, that is 
a \emph{time-averaged} heat kernel.
\begin{definition}
  Let us denote with $q_t:M\times M\to\R$ the unique function with null mean value such that 
  $-\lapl q_t(x,y) = p_t(x,y)-1$, where the Laplacian is computed with respect to the second variable.
\end{definition}

\begin{remark}
  All the derivatives of $q_t$ will be performed on the second variable.
  In the weighted and with boundary setting the definition of $q_t$ stays the same, 
  whereas the Laplace operator changes meaning (as explained in \cref{sub:setting}).
\end{remark}
\begin{remark}
  The kernel $q_t$ arises naturally in our investigation because, as we will see later, it holds
  \begin{equation*}
    f^{n,t}(y) = \frac1n\sum_{i=1}^n q_t(X_i, y)\fullstop
  \end{equation*}
\end{remark}

Let us show some properties of the kernel $q_t$.
As a consequence of the decay of the heat kernel when $t$ goes to infinity, for any $x,y\in M$ and $t>0$, it holds
\begin{equation*}
    \lapl_y \int_t^{\infty} p_s(x,y)-1\de s = \int_t^{\infty} \lapl_y p_s(x,y)\de s
    = \int_t^{\infty} \frac{\de}{\de s} p_s(x,y)\de s = 1 - p_t(x,y) \comma
\end{equation*}
therefore we have the fundamental identity
\begin{equation}\label{eq:q_expr}
  q_t(x,y) = \int_t^{\infty}\left(p_s(x,y)-1\right)\de s \fullstop
\end{equation}
Let us also remark that $q_t$ is symmetric $q_t(x,y) = q_t(y, x)$. 
Furthermore, for all $y\in M$ the average value of $\nabla_y q_t(\emptyparam, y)$ is null, indeed it holds
\begin{align*}
  \int_M \nabla_y q_t(x, y)\de\m(x) 
  &= \int_M \nabla_y\left(\int_t^{\infty}\left(p_s(x,y)-1\right)\de s\right)\de\m(x) \\
  &= \int_t^{\infty}\nabla_y\left(\int_M p_s(x,y)\de\m(x)\right)\de s
  = \int_t^{\infty}\nabla_y(1) \de s = 0 \fullstop
\end{align*}

Similarly we can prove that the average value is null also for higher derivatives.

Now we want to deduce some estimates for the time-averaged kernel $q_t$ from the related estimates for
the standard heat kernel.
Therefore let us start stating some well-known estimates related to the heat kernel. The interested reader
can find more about heat kernel estimates on the monographs \cite{saloff2010,grigor1999}.

\begin{theorem}[Trace Formula]\label{thm:heat_ondiag}
  It holds
  \begin{equation*}
    \int_M \left(p_t(x,x)-1\right)\de\m(x) = \frac 1{4\pi t} + \bigo\left(\frac 1{\sqrt{t}}\right)\fullstop
  \end{equation*}
\end{theorem}
\begin{proof}
  It is proved in \cite{McKeanSinger67} for smooth manifolds, possibly with smooth boundary.
  To handle the case of the square, we need the formula for Lipschitz domains that is proved
  in \cite{brown1993}.
  For weighted manifolds it is proven in \cite[Theorem 1.5]{charalambous2017}.
\end{proof}

\begin{theorem}[Heat Kernel Estimate]\label{thm:heat_estimate}
  There exists a suitable $a=a(M)>1$ such that, for any $0<t<1$ and $x,y\in M$, it holds
  \begin{equation*}
    p_t(x,y) \lesssim \frac 1t \,a^{-\frac{d^2(x,y)}{t}} \fullstop
  \end{equation*}
\end{theorem}
\begin{proof}
    It is proved in \cite[Theorem 4, Remark at page 1050]{ChengLiYau81}.
    For the proof in the weighted and with boundary setting, see \cite[Theorem 1.2]{Jiang-Li-Zhang15}.
\end{proof}

\begin{theorem}[Heat Kernel Derivatives Estimate]\label{thm:heat_deriv}
  For any $N\ge 1$, $0<t<1$ and $x,y\in M$, it holds
  \begin{equation*}
    \abs{\nabla^N p_t(x,y)} \lesssim_N
    \left(\frac 1{t^{\frac N2}} + \frac{d^N(x,y)}{t^N}\right)p_t(x,y) \fullstop
  \end{equation*}
\end{theorem}
\begin{proof}
  For closed compact manifolds it can be found in \cite{stroock1998} or in \cite[Corollary 1.2]{Hsu99}.
  We prove the special case $M=\cc01^2$ in \cref{app:square}.
\end{proof}
\begin{remark}
  The estimate on the derivatives of the heat kernel provided by the previous theorem is 
  fundamental for the approach presented here. 
  Furthermore, as anticipated in \cref{sub:setting}, the need for such an estimate is 
  exactly the obstruction to the generalization of our result to the weighted or with boundary setting.
  
  Let us also remark that we will use the estimate only for $N\le 3$.
\end{remark}

We are now ready to state and prove some estimates on the kernel $q_t$. 
The first one has an algebraic flavor, whereas \cref{prop:q_deriv,cor:q_four} are 
hard-analysis estimates deduced from \cref{thm:heat_estimate,thm:heat_deriv}.

\begin{proposition}\label{prop:on_diag_int}
  For any $t>0$ it holds
  \begin{equation*}
    \int_M\int_M \abs{\nabla_y q_t(x,y)}^2\de\m(y)\de\m(x) 
    %= \int_{2t}^{\infty}\int_M \left(p_s(x,x)-1\right)\de\m(x)\de s
    = \frac{\abs{\log(t)}}{4\pi} + \bigo(1)\fullstop
  \end{equation*}
\end{proposition}
\begin{proof}
  Let us ignore the integral in $\de\m(x)$.
  Recalling the formula stated in \cref{eq:q_expr} and the definition of $q_t$, integrating by parts we obtain
  \begin{equation*}
    \int_M \abs{\nabla_y q_t(x,y)}^2\de\m(y) 
    = \int_M (p_t(x,y)-1)\int_t^{\infty}(p_s(x,y)-1)\de s\de\m(y) \comma
  \end{equation*}
  thence, applying Fubini's theorem and the semigroup property of $p_t$, we can continue this 
  chain of identities
  \begin{equation*}
    =\int_t^{\infty}\int_M (p_t(x,y)p_s(y, x)-1)\de\m(y)\de s
    =\int_t^{\infty}(p_{s+t}(x,x)-1)\de s
    =\int_{2t}^{\infty}(p_s(x,x)-1)\de s \fullstop
  \end{equation*}
  After integration with respect to $x$, the statement follows thanks to \cref{thm:heat_ondiag}.
\end{proof}

The following proposition plays a central role in all forthcoming results. Indeed, the kind of 
estimate we obtain on the derivatives of $q_t$ is exactly the one we need to deduce strong 
integral inequalities.
\begin{proposition}\label{prop:q_deriv}
  For any $N\ge 1$, $0<t<1$ and $x,y\in M$ it holds
  \begin{equation*}
    \abs{\nabla^N_y q_t(x,y)} \lesssim_N \frac{1}{d^N(x,y)+t^{\frac N2}} \fullstop
  \end{equation*}
\end{proposition}
\begin{proof}
  For the sake of brevity, let us denote $d=d(x,y)$.

  Applying \cref{eq:q_expr} together with \cref{thm:heat_estimate,thm:heat_deriv}, through some 
  careful estimates of the involved quantities and the change of variables $d^2/s=w$, we obtain
  \begin{align*}
    \abs{\nabla^N_y q_t(x,y)}
    &\le \int_t^{\infty} \abs{\nabla^N p_s(x,y)}\de s
    \lesssim_N 1 + \int_t^1\left(\frac{1}{s^{\frac{N}2}} 
    + \left(\frac ds\right)^N\right)\frac 1s\, a^{-\frac{d^2}s}\de s \\
    &\le 1 + \int_t^\infty\left(\frac{1}{s^{\frac{N}2}} 
    + \left(\frac ds\right)^N\right)\frac 1s\, a^{-\frac{d^2}s}\de s \\
    &= 1 + \frac 1{d^N}\int_0^{\frac{d^2}t}\left(w^{\frac N2-1}+w^{N-1}\right)a^{-w}\de w \\
    &= 1 + \frac 1{t^{\frac N2}}\cdot
    \frac{\int_0^{\frac{d^2}t}\left(w^{\frac N2-1}+w^{N-1}\right)a^{-w}\de w}
    {\left(\frac{d^2}t\right)^{\frac N2}} \fullstop
  \end{align*}
  The desired statement now follows from the elementary inequality
  \begin{equation*}
      \forall\, x > 0:\quad \int_0^x (w^{\frac N2-1} + w^{N-1})a^{-w}\de w \lesssim_N \frac{x^{\frac N2}}{1+x^{\frac N2}} \comma
  \end{equation*}
  which is a consequence of the two estimates
  \begin{align*}
    \forall\, 0 < x < 1:\quad &\int_0^x (w^{\frac N2-1} + w^{N-1})a^{-w}\de w \lesssim_N x^{\frac N2} \comma\\
    \forall\, x \ge 1:\quad &\int_0^x (w^{\frac N2-1} + w^{N-1})a^{-w}\de w \lesssim_N 1 \fullstop
  \end{align*}
\end{proof}

\begin{corollary}\label{cor:q_four}
  Let us fix a natural number $N\ge 1$ and a real number $p > \frac2N$. 
  For any $0<t<1$ and $\bar x,\bar y\in M$ the following two inequalities hold\footnote{Let us remark that the
  two inequalities are not equivalent. Indeed in the first one we are integrating with respect to the variable $y$, that
  is the differentiation variable, whereas in the second one we are integrating with respect to the variable $x$, that 
  \emph{is not} the differentiation variable.}
  \begin{align*}
    &\int_M\abs{\nabla^N_y q_t(\bar x,y)}^p\de\m(y) \lesssim_{N,p} t^{1-\frac{Np}2}\comma \\
    &\int_M\abs{\nabla^N_y q_t(x,\bar y)}^p\de\m(x) \lesssim_{N,p} t^{1-\frac{Np}2}\fullstop
  \end{align*}
  When $p=\frac2N$, the same inequalities hold if $t^{1-\frac{Np}2}$ is replaced with $\abs{\log(t)}$.
\end{corollary}
\begin{proof} 
  We are going to prove only the case $Np>2$ when the integral is in $\de\m(y)$, the proof of
  the other case being very similar.

  To prove the desired result we just have to insert the inequality stated in \cref{prop:q_deriv} 
  inside the coarea formula (in the very last inequality below we use the change of variables
  $r^2=st$):
  \begin{align*}
    \int_M\abs{\nabla^N_y q_t(x,y)}^p\de\m(y) 
    &\lesssim_{N,p} \int_M\frac{1}{(d^N(x,y)+t^{N/2})^p}\de\m(y) \\
    &\lesssim_{N,p} \int_{B(x,\injradius(M)/2)} \frac{1}{(d^N(x,y)+t^{N/2})^p}\de\m(y) \\
    &= \int_0^{\injradius(M)/2}\frac{1}{(r^N+t^{N/2})^p}\Haus^1(\partial B(x, r))\de r \\
    &\lesssim \int_0^{\injradius(M)/2}\frac{r}{(r^N+t^{N/2})^p}\de r
    \lesssim_{N,p} t^{1-\frac{Np}2} \fullstop
  \end{align*}
  Let us recall that $\Haus^1$ is the Hausdorff measure induced by the Riemannian distance. 
  In our inequalities we have implicitly applied the known estimate (see \cite{Petersen98}) 
  on the measure of the spheres on a smooth Riemannian manifold. It would be possible to obtain 
  the same result applying only the estimate on the measure of the balls (that is a somewhat 
  more elementary inequality) and Cavalieri's principle instead of the coarea formula.
\end{proof}

We are going to transform the results we have proven about $q_t$ into inequalities concerning $f^{n,t}$.

As anticipated, it holds
\begin{equation}\label{eq:f_as_average}
  f^{n,t}(y) = \frac1n\sum_{i=1}^n q_t(X_i, y)\fullstop
\end{equation}
Such an identity can be proven showing that both sides have null mean and the same Laplacian:
\begin{align*}
  -\lapl f^{n,t}(y) = u^{n,t}(y)-1 &= P^*_t\left(\frac1n \sum_{i=1}^n \delta_{X_i}-1\right)(y) 
  = \frac1n\sum_{i=1}^n p_t(X_i,y)-1 \\
  &= -\lapl\left(\frac1n\sum_{i=1}^n q_t(X_i, \emptyparam)\right)(y) \fullstop
\end{align*}

\begin{lemma}\label{lem:int_nabla_f}
  The following approximation holds
  \begin{equation*}
    \E{\int_M\abs{\nabla f^{n,t}}^2\de\m} 
    = \frac{\abs{\log(t)}}{4\pi n} + \bigo\left(\frac{1}n\right) \fullstop
  \end{equation*}
\end{lemma}
\begin{proof}
  Using the linearity of the expected value and the independence of the variables $X_i$, we obtain
  \begin{equation*}
    \E{\int_M\abs{\nabla f^{n,t}}^2\de\m(y)} =\frac1n\int_M\int_M\abs{\nabla_y q_t}^2\de\m(y)\de\m(x)
  \end{equation*}
  and therefore, applying \cref{prop:on_diag_int}, we have proven the desired approximation.
\end{proof}
\begin{remark}
  The previous lemma is contained in \cite[Lemma 3.16]{Ambrosio-Stra-Trevisan2018}.
\end{remark}

Our next goal is proving \cref{thm:prob_flat_density}. 
In order to do so we will need Bernstein inequality for the sum of independent and identically 
distributed random variables. 
We are going to state it here exactly in the form we will use. 
One of the first reference containing the inequality is \cite{Bernstein46}, whereas one of the
first \emph{in English} is \cite{Hoeffding63}. 
A more recent exposition can be found in the survey \cite[Theorem 3.6,3.7]{ChungLu2006} (alternatively, see the monograph \cite{boucheron2003}).
\begin{theorem}[Bernstein Inequality]\label{thm:bernstein}
  Let $X$ be a centered random variable such that, for a certain constant $L>0$, 
  it holds $\abs{X}\leq L$ almost surely.
  If $(X_i)_{1\le i\le n}$ are independent random variables with the same distribution as $X$, 
  for any $\xi>0$ it holds
  \begin{equation*}
    \P{\abs*{\frac{\sum_{i=1}^n X_i}{n}}>\xi}
    \le 2\exp\left(-\frac{n\xi^2}{2\E{X^2}+\frac23L\xi}\right) \fullstop
  \end{equation*}
\end{theorem}

\begin{proof}[Proof of \cref{thm:prob_flat_density}]
  Our strategy is to gain, for any $y\in M$, a pointwise bound on the probability 
  $\P{\abs{\nabla^2f^{n,t}(y)}>\xi}$ through the aforementioned concentration inequality and then to achieve 
  the full result using a sufficiently fine net of points on $M$.

  Let us fix $y\in M$. Recalling \cref{eq:f_as_average}, we can apply \cref{thm:bernstein} 
  in conjunction with \cref{prop:q_deriv,cor:q_four} and obtain\footnote{Note that we are applying \cref{thm:bernstein} to matrix-valued random variables. In order to prove this generalization it is sufficient to apply the theorem to each entry of the matrix.}
  \begin{align*}
    \P{\abs{\nabla^2 f^{n,t}(y)}>\xi}
    &\lesssim \exp\left(-\frac{n\xi^2}{
    2\norm{\nabla^2 q_t(\emptyparam, y)}^2_2+\frac23\norm{\nabla^2 q_t(\emptyparam, y)}_\infty\xi
    }\right) \\
    &\lesssim \exp\left(-nCt\xi^2\right)\comma
  \end{align*}
  where $C=C(M)$ is a suitable constant.

  Let us now fix $\ell>0$ such that 
  $\ell\cdot L=\frac\xi2$, where $L$ is the Lipschitz constant of $\abs{\nabla^2 f}$. 
  Let $(z_i)_{1\le i\le N(\ell)}\subseteq M$ be 
  an $\ell$-net on the manifold. 
  Through a fairly standard approach, we can find such a net with $N(\ell)\lesssim \ell^{-2}$.
  Furthermore, it is easy to see, considering the bound imposed on $\ell$, that if 
  $\abs{\nabla^2 f}(z_i)<\xi/2$ for any $1\le i\le N(\ell)$, then $\norm{\nabla^2 f}_\infty<\xi$.
  Therefore we have
  \begin{align*}
    \P{\norm{\nabla^2 f^{n,t}}_\infty>\xi}
    &\le \sum_{i=1}^{N(\ell)}\P{\abs{\nabla^2 f^{n,t}}(z_i)>\xi/2}\\
    &\lesssim \frac1{\ell^2} \exp\left(-\frac{nCt\xi^2}4\right) 
    = \frac{4L^2}{\xi^2}\exp\left(-\frac{nCt\xi^2}{4}\right)\fullstop
  \end{align*}
  The statement follows noticing that, thanks to \cref{eq:f_as_average}, $L$ can be bounded 
  from above by\footnote{The definition of the infinity norm for a $3$-tensor 
  is analogous to the definition given for $2$-tensors in \cref{def:tensor_infinity_norm}.} 
  \begin{equation*}
    \max_{x\in M}\norm{\nabla^3 q_t(x,\cdot)}_\infty\lesssim t^{-\frac32} \comma
  \end{equation*}
  where we used \cref{prop:q_deriv} in the last inequality.
\end{proof}
\begin{remark}\label{rem:complement_event_is_negligible}
  Let us stress that if $t\gg\frac{\log n}{n}$, 
  \cref{thm:prob_flat_density} shows that \emph{quite often} the hessian of $f^{n,t}$ is \emph{very small}.

  More precisely, if $\xi_n\ge\frac 1n$ and $t_n = \log(a)^{-1}\kappa_n \frac{\log(n)}{n}$ with $\kappa_n\ge 1$, then
  \begin{equation*}
    \P{(A^{n,t_n}_{\xi_n})^{\complement}} \lesssim \frac{1}{n^{\kappa_n\xi_n^2-5}} \fullstop
  \end{equation*}
  This observation will play a major role in the last sections as it will
  allow us to ignore completely the complementary event $A^{n,t}_\xi$ as it is sufficiently small.
\end{remark}

\section{Transport Cost Inequality}\label{sec:transport_ineq}
Given two density functions $u_0,u_1:M\to\R$, the Dacorogna-Moser coupling (see \cite[16-17]{Villani09}) gives
us a ``first-order'' approximation of the optimal matching between them.
With this technique in mind, following the ideas developed in \cite{Ambrosio-Stra-Trevisan2018,Ledoux2017}, 
we are presenting a proposition that generalizes some of the results mentioned in those two papers.
Its proof is simpler than those presented in \cite{Ambrosio-Stra-Trevisan2018,Ledoux2017} 
but relies on the Benamou-Brenier formula (see \cite{Benamou-Brenier00}). 
Let us remark that our application of the Benamou-Brenier formula is somewhat an overkill, 
nonetheless we believe that using it makes the result much more natural.

For the ease of the reader let us recall the said formula before stating the proposition and its proof.
\begin{definition}[Flow plan]
  Given two measures $\mu_0,\mu_1\in\prob(M)$, a flow plan is the joint information of a weakly 
  continuous curve of measures $\mu_t\in C^0(\cc01, \prob(M))$ and a time-dependent Borel 
  vector field $(v_t)_{t\in \oo01}$ on $M$ such that, in $\oo01\times M$, 
  the continuity equation holds in the  distributional sense
  \begin{equation*}
    \frac{\de}{\de t}\mu_t+\div(v_t\mu_t)=0\fullstop
  \end{equation*}
\end{definition}

\begin{theorem}[Benamou-Brenier formula]
  Given two measures $\mu_0,\mu_1\in\prob(M)$, it holds
  \begin{equation*}
    W_2^2(\mu_0, \mu_1) =
    \inf\left\{\int_0^1\int_M \abs{v_t}^2\de\mu_t \de t\st (\mu_t, v_t)\text{ is a flow plan 
    between $\mu_0$ and $\mu_1$}\right\}\fullstop
  \end{equation*}
  In addition, if $(\mu_t, v_t)$ is a flow plan with $\norm{v_t}_\infty+\norm{\nabla v_t}_\infty\in L^\infty(0,1)$, the flow at time $1$ induced by the vector 
  field $v_t$ is a transport map between $\mu_0$ and $\mu_1$ and its cost can be estimated by
  \begin{equation*}
    \int_0^1\int_M \abs{v_t}^2\de\mu_t \de t\fullstop
  \end{equation*}
\end{theorem}

\begin{proposition}\label{prop:transport_bound}
  Given two positive, smooth density functions $u_0, u_1$ in $M$, let $f\in C^\infty(M)$ be the unique solution 
  of $-\lapl f = u_1-u_0$ with null mean.
  For any increasing function $\theta\in C^1(\cc01)$ such that $\theta(0)=0$ and $\theta(1)=1$, it holds
  \begin{equation*}
    W_2^2(u_0\m, u_1\m)
    \le \int_M\abs{\nabla f}^2
    \left(\int_0^1\frac{\theta'(t)^2}{u_0(1-\theta(t))+u_1\theta(t)}\de t\right)
    \de\m \fullstop
  \end{equation*}
  Furthermore a map that realizes the cost at the right hand side is the flow at time $1$ induced 
  by the time-dependent vector field
  \begin{equation*}
    v_t(x) = \frac{\theta'(t)\nabla f(x)}{u_0(1-\theta(t))+u_1\theta(t)} \fullstop
  \end{equation*}
\end{proposition}
\begin{proof}
  Let us define the convex combination $u_t\defeq (1-\theta(t))u_0+\theta(t)u_1$. 
  It is straightforward to check
  \begin{equation*}
    \frac{\de}{\de t}u_t(x) 
    + \div\left(u_t(x)\cdot\left(\frac{\theta'(t)}{u_t(x)}\nabla f(x)\right)\right) = 0 \comma
  \end{equation*}
  hence, thanks to the Benamou-Brenier formula, we have
  \begin{equation*}
    W_2^2(u_0\m, u_1\m)\le 
    \int_0^1 \int_M \frac{\theta'(t)^2}{u_t^2}\abs{\nabla f}^2 u_t\de\m\de t 
    = \int_M\abs{\nabla f}^2\int_0^1\frac{\theta'(t)^2}{u_t}\de t\de\m \fullstop
  \end{equation*}
\end{proof}

\begin{corollary}
  Given two smooth, positive density functions $u_0, u_1$ in $M$, let $f\in C^\infty(M)$ 
  be the unique solution of $-\lapl f = u_1-u_0$ with null mean.
  It holds
  \begin{equation}\label{eq:shallow_transport_ineq}
    W_2^2(u_0\m, u_1\m)\le 4 \int_M \frac{\abs{\nabla f}^2}{u_0}\de\m
  \end{equation}
  and also
  \begin{equation}\label{eq:sharp_transport_ineq}
    W_2^2(u_0\m, u_1\m)\le \int_M \abs{\nabla f}^2\frac{\log(u_1)-\log(u_0)}{u_1-u_0}\de\m \comma
  \end{equation}
  where the ratio is understood to be equal to $1$ if $u_1=u_0$.
\end{corollary}
\begin{proof}
  The inequality \cref{eq:shallow_transport_ineq} would follow from \cref{prop:transport_bound} 
  if we were able to find a function $\theta$ such that for any $x\in M$ it holds
  \begin{equation}\label{eq:transport_elementary_ineq}
    \int_0^1\frac{\theta'(t)^2}{u_0(x)(1-\theta(t))+u_1(x)\theta(t)}\de t \le \frac{4}{u_0(x)}\fullstop
  \end{equation}
  Let us start with the observation
  \begin{equation*}
    u_0(1-\theta(t))+u_1\theta(t) \ge u_0(1-\theta(t))
  \end{equation*}
  and therefore, in order to get \cref{eq:transport_elementary_ineq}, it suffices to have
  \begin{equation*}
    \int_0^1\frac{\theta'(t)^2}{1-\theta(t)}\de t \le 4
  \end{equation*}
  that is satisfied with equality by $\theta(t) = 1-(1-t)^2$.

  The inequality \cref{eq:sharp_transport_ineq} follows from \cref{prop:transport_bound} 
  choosing $\theta(t)=t$ and computing the definite integral.
\end{proof}
\begin{remark}
  The inequality \cref{eq:shallow_transport_ineq} will be used when we can control only one 
  of the two densities involved, whereas the sharper \cref{eq:sharp_transport_ineq} will be used 
  when we can show that both densities are already very close to $1$.
\end{remark}
\begin{remark}
    On a compact Riemannian manifold, it is well-known that the Wasserstein distance $W_2^2(u_0\m, u_1\m)$ is controlled by the relative entropy and by the Fisher information (see \cite{Otto-Villani00}). In our setting this kind of inequalities is not sufficient, as we need estimates that are \emph{almost-sharp} when $u_0, u_1\approx 1$. Moreover, both the relative entropy and the Fisher information depend on the pointwise value of the relative density $\frac{u_1}{u_0}$, but we need estimates that depend nonlocally on the density (as we must take into account the \emph{macroscopical} differences between $u_0$ and $u_1$). On the other hand, \cref{eq:shallow_transport_ineq,eq:sharp_transport_ineq} control the Wasserstein distance with the negative Sobolev norm $\norm{\frac{u_1}{u_0}-1}_{H^{-1}}$, that is nonlocal, and \cref{eq:sharp_transport_ineq} is \emph{almost-sharp} when $u_0,u_1\approx 1$.
\end{remark}

\section{Refined contractivity of the Heat Flow}\label{sec:extreme_contractivity}
The following proposition (see for example \cite[Theorem 3]{ErbarKuwadaSturm2015}) is the well-known H\"older continuity of the heat flow with respect 
to the Wasserstein  distance.
\begin{proposition}\label{prop:heat_holder}
  Given a measure $\mu\in \prob(M)$ and a positive $t>0$, it holds
  \begin{equation*}
    W_2^2(\mu, P_t^*(\mu)) \lesssim t \fullstop
  \end{equation*}
\end{proposition}
In this section we are going to prove a more refined (asymptotic) contractivity result, when $\mu$ is
an empirical measure built from an i.i.d. family with law $\m$.

The following theorem proves that the estimate described in 
\cref{prop:heat_holder} is far from being sharp when the measure $\mu$ is the empirical measure 
generated by $n$ random points. Indeed it shows that the average growth of the Wasserstein distance squared
is not linear after the threshold $t = \log\log(n)/n$.

Such a trend would be expected if the matching cost had magnitude $\bigo(\log\log(n)/n)$, but its magnitude is
$\bigo(\log(n)/n)$. This quirk shall be seen as a manifestation of the fact that, in dimension $2$, the
obstructions to the matching are both the \emph{global} and \emph{local} discrepancies in distribution 
between the empirical measure and the reference measure (see \cite[Section 4.2]{Talagrand14} for further
details on this intuition). Regularizing with the heat kernel we take into account only 
the short-scale discrepancies and thus we stop observing linear growth way before the real matching cost is
achieved.

Given that in higher dimension the main obstruction to the matching is concentrated in the 
microscopic scale, we don't think that a similar statement can hold in dimension $d>2$. 
With the wording ``similar statement'' we mean the fact that $\E{W_2^2(\mu^{n,t}, \m)}\ll t$ when
$t=\bigo(n^{-d/2})$ (let us recall that, if $d>2$, the expected matching
cost has order $\bigo(n^{-d/2})$).
\begin{theorem}\label{thm:near_stability_heat_kernel}
  Given a positive integer $n\in\N$, let $(X_i)_{1\le i\le n}$ be $n$ independent random points
  $\m$-uniformly distributed on $M$.
  Let $\mu^n = \frac1n\sum_{i=1}^n\delta_{X_i}$ be the empirical measure associated to the 
  points $(X_i)_{1\le i\le n}$. There exists a constant $C=C(M)>0$ such that, for any 
  time $t=\alpha/n$ with $\alpha\ge C\log(n)$, 
  denoting $\mu^{n,t}=P_t^*(\mu^n)$, it holds
  \begin{equation*}
    \E{W_2^2(\mu^n, \mu^{n,t})} \lesssim 
    \frac{\log(\alpha)}{n} \quad\left(\,\ll \frac{\alpha}{n}=t \,\right)\fullstop
  \end{equation*}
\end{theorem}
\begin{proof}
  Our approach consists of using \cref{prop:heat_holder} to estimate the distance between $\mu^n$ and 
  $\mu^{n,1/n}$ and then adopting \cref{prop:transport_bound} to estimate the distance 
  from $\mu^{n,1/n}$ to $\mu^{n,t}$.

  Let $u^{n,s}=\frac 1n\sum_1^np_s(X_i,\emptyparam)$ be the density of $P_s^*(\mu^n)$ and
  let us fix the time $t_0=\frac1n$. 

  Recalling \cref{prop:heat_holder}, it holds
  \begin{align*}
    \E{W_2^2(\mu^n, \mu^{n,t})}&\le 2\E{W_2^2(\mu^n, \mu^{n,t_0})} + 2\E{W_2^2(\mu^{n,t_0}, \mu^{n,t})} \\
    &\lesssim \frac1n + \E{W_2^2(\mu^{n,t_0}, \mu^{n,t})}
  \end{align*}
  In order to bound $\E{W_2^2(\mu^{n,t_0}, \mu^{n,t})}$, let us restrict our study to the event 
  $A^{n,t}_{\frac12}$.
  As stated in \cref{thm:prob_flat_density}, such an event is so likely (as a consequence of the
  assumption $\alpha\gg\log(n)$) that its complement can be completely ignored because all 
  quantities that we are estimating have polynomial growth in $n$ (recall 
  \cref{rem:complement_event_is_negligible}).

  Let us denote $f:M\to\R$ the null mean solution to the Poisson equation $-\lapl f = u^{n,t_0}-u^{n,t}$, 
  representable as $\int_{t_0}^t u^{n,s}-1\de s$.
  Recalling that we are in the event $A^{n,t}_{\frac12}$, we can apply \cref{eq:shallow_transport_ineq} 
  and obtain
  \begin{equation*}
    W_2^2(\mu^{n,t_0}, \mu^{n,t})\le 
    4\int_M \frac{\abs{\nabla f}^2}{u^{n,t}}\de\m \lesssim \int_M \abs{\nabla f}^2\de\m \fullstop
  \end{equation*}
  Using the independence of $p_a(X_i,y)$ and $p_b(X_j,y)$ for $a,b>0$, $y\in M$ and $i\neq j$, 
  we are now able to compute the expected value
  \begin{align*}
    \E{W_2^2(\mu^{n,t_0}, \mu^{n,t})} &\lesssim 
    \E{\int_M \abs{\nabla f}^2\de\m}  
    = \E{\int_M -\lapl f \cdot f\de\m} \\
    &= \E{\int_M (u^{n,t_0}-u^{n,t})\left(\int_{t_0}^t u^{n,s}-1\de s\right)\de\m} \\
    &=\frac1n\int_{t_0}^t\E{\int_M (p_{t_0}(X, y)-p_t(X,y))p_s(X,y)\de\m(y)} \de s \\
    &=\frac1n\int_{t_0}^t\int_M (p_{t_0+s}(x, x)-p_{t+s}(x,x))\de\m(x) \de s\\
    &\le \frac1n\int_{2t_0}^{t+t_0}\int_M (p_{s}(x, x)-1)\de\m(x) \de s
    \lesssim\frac1n\int_{2t_0}^{t+t_0}\frac1s \de s 
    \lesssim \frac{\log(\alpha)}{n}
  \end{align*}
  that is exactly the desired result. Let us remark that in one of the inequalities we have exploited
  the bound $p_r(x, x )\ge 1$ with $r=t+s$, a simple consequence of the Chapman-Kolmogorov property.
\end{proof}

\begin{remark}\label{rem:easy_restriction_rule}
  Before going on, let us take a minute to isolate and describe the approach we have employed to
  restrict our study to the event $A^{n,t}_{\frac12}$.

  Let $X, Y$ be random variables such that $X\equiv Y$ in an event $A$. It holds
  \begin{equation*}
    \abs*{\E{X}-\E{Y}}
    \le (\norm{X}_\infty+\norm{Y}_\infty)\P{A^\complement} \fullstop
  \end{equation*}
  Therefore, exactly as we did in the previous proof, if the event $A$ is \emph{much} smaller than the inverse
  of the magnitude of $X$ and $Y$, we can safely exchange $X$ and $Y$ when computing expected values.
\end{remark}
\begin{remark}
  In order to exploit \cref{rem:easy_restriction_rule} in the proofs of 
  \cref{thm:near_stability_heat_kernel,thm:main_theorem_semidiscrete,thm:main_theorem_bipartite}, 
  it is necessary to check that all involved quantities have at most polynomial growth in $n$ 
  (see \cref{rem:complement_event_is_negligible}).
  
  Let us prove, for example, that $\int_M \abs{\nabla f^{n,t}}^2\de\m$ has polynomial 
  growth in $n$ whenever $t=t(n)\ge \frac1n$. Thanks to standard elliptic estimates, it holds
  \begin{equation*}
    \norm{\nabla f^{n,t}}_2 \lesssim \norm{\lapl f^{n,t}}_2 = \norm{u^{n,t}-1}_2
    \le \sup_{x\in M}\norm{p_t(x,\emptyparam)}_2 \lesssim t^{-1}\comma
  \end{equation*}
  where in the last inequality we have applied \cref{thm:heat_estimate}. Thence, the desired 
  control over $\int_M \abs{\nabla f^{n,t}}^2\de\m$ follows from the condition on $t=t(n)$.
  
  For the proof of \cref{thm:main_theorem_semidiscrete} this turns out to be sufficient, whereas
  for the proofs of \cref{thm:near_stability_heat_kernel,thm:main_theorem_bipartite} some similar
  (but not identical) quantities should be controlled. 
  We do not write explicitly how to control them as the exact same reasoning works with minor
  changes.
\end{remark}

\section{Semi-discrete Matching}\label{sec:semidiscrete_matching}
This section is devoted to the computation, with an asymptotic estimate of the error term, of the average 
matching cost between the empirical measure generated by $n$ random points and the reference measure.

An estimate of the error term was recently provided by Ledoux in 
\cite[Eqs. (16) and (17)]{Ledoux18}.
Our estimate is slightly better than the one proposed by Ledoux; indeed he estimates the error as 
$\bigo(\log\log(n)\log(n)^{3/4}/n)$ whereas our estimate is $\bigo(\sqrt{\log\log(n)\log(n)}/n)$

Let us briefly sketch the strategy of the proof.
\begin{description}
  \item[Step 1] The inequality developed in the previous section allows us to choose $t$ of magnitude 
  $\bigo(\log^3(n)/n)$ while keeping $W^2_2(\mu^n,\mu^{n,t})$ under strict control. With such a choice
  of the regularization time, we can apply \cref{thm:prob_flat_density} and get that $A^{n,t}_\xi$ is a very likely event.
  Without \cref{thm:near_stability_heat_kernel} we would have been able only to choose 
  $t=\smallo(\log(n)/n)$ and that would have invalidated the proof.
  \item[Step 2] Using \cref{eq:sharp_transport_ineq} we estimate the matching cost between $\mu^{n,t}$ 
  and $\m$. It comes out that, in the event $A^{n,t}_\xi$, this matching cost is almost equal to
  \begin{equation*}
    \int_M \abs{\nabla f^{n,t}}^2\de\m
  \end{equation*}
  and we are able to evaluate it thanks to \cref{lem:int_nabla_f}.
\end{description}

The statement we are giving here is slightly stronger than the statement given in the introduction (as we
can now use the function $f^{n,t}$).
\begin{reptheorem}{thm:main_theorem_semidiscrete}
  Let $(M,\metric)$ be a $2$-dimensional compact closed manifold (or the square $\cc01^2$) whose volume 
  measure $\m$ is a probability. 
  Let $(X_i)_{i\in\N}$ be a family of independent random points $\m$-uniformly distributed on $M$.
  For a suitable choice of the constant $\gamma>0$, setting $t(n) = \gamma\frac{\log^3 n}{n}$, it holds
  \begin{equation*}
    \E{\abs*{W_2\left(\frac1n\sum_{i=1}^n\delta_{X_i}, \m\right) 
    - \sqrt{\int \abs{\nabla f^{n,t(n)}}^2\de\m}}} 
    \lesssim \sqrt{\frac{\log\log(n)}n} \fullstop
  \end{equation*}
  Furthermore it also holds
  \begin{equation*}
    \E{\abs*{W^2_2\left(\frac1n\sum_{i=1}^n\delta_{X_i}, \m\right) - \int \abs{\nabla f^{n,t(n)}}^2\de\m}}
    \lesssim \frac{\sqrt{\log(n)\log\log(n)}}n \comma
  \end{equation*}
  from which follows
  \begin{equation*}
    \lim_{n\to\infty} \frac{n}{\log(n)}\cdot\E{W^2_2\left(\frac1n\sum_{i=1}^n\delta_{X_i}, \m\right)}
    =\frac1{4\pi} \fullstop
  \end{equation*}
\end{reptheorem}
\begin{proof}
  Let us fix a parameter $\xi=\frac1{\log n}$ and the time variable 
  $t=\log(a)^{-1}\gamma\frac{\log^3 n}{n}$, so that \cref{rem:complement_event_is_negligible} gives
  $$
  \P{(A^{n,t}_{\xi})^{\complement}} \lesssim \frac{1}{n^{\gamma-5}}\fullstop
  $$
  Hence, by choosing $\gamma>0$ sufficiently large, we can obtain any power-like decay we need.
  
  Recalling \cref{thm:near_stability_heat_kernel}, we obtain
  \begin{equation}\label{eq:mu_n_to_mu_nt}
    \E{W^2_2(\mu^n,\mu^{n,t})} \lesssim \frac{\log\log(n)}n \fullstop
  \end{equation}

  Thanks to \cref{prop:transport_bound} we know that $\mu^{n,t}$ is the push-forward of $\m$ through the
  flow at time $1$ of the time-dependent vector field
  \begin{equation*}
    Y_s = \frac{\nabla f^{n,t}}{1 + s(u^{n,t}-1)} \fullstop
  \end{equation*}
  Thus, if we assume to be in the event $A^{n,t}_\xi$ with $n$ sufficiently large, 
  we can apply \cref{prop:stability_flow} with $X=\nabla f^{n,t}$ and 
  $Y_s=\frac{\nabla f^{n,t}}{1 + s(u^{n,t}-1)}$ to obtain 
  \begin{equation}\label{eq:flow_to_exp}
    W^2_2(\mu^{n,t}, \exp(\nabla f^{n,t})_\#\m)
    \lesssim
    \xi^2\int_M \abs{\nabla f^{n,t}}^2\de\m \fullstop
  \end{equation}
  Still working in the event $A^{n,t}_\xi$, thanks to \cite[Theorem 1.1]{Glaudo19}, we can say 
  (for a sufficiently large $n$) that
  \begin{equation}\label{eq:exp_to_formula}
    W^2_2(\exp(\nabla f^{n,t})_\#\m, \m) = \int_M \abs{\nabla f^{n,t}}^2 \de\m\fullstop
  \end{equation}

  Once again, as we have done in the proof of \cref{thm:near_stability_heat_kernel}, 
  let us notice that the restriction of our analysis to the event $A^{n,t}_\xi$ is not an issue. 
  Indeed its complement is so small that, 
  because all the quantities involved have no more than polynomial growth in $n$, using
  the approach described in \cref{rem:easy_restriction_rule}, we can restrict our study to the
  event $A^{n,t}_\xi$ thanks to \cref{thm:prob_flat_density}.

  Hence, joining \cref{eq:mu_n_to_mu_nt}, \cref{eq:flow_to_exp} and \cref{eq:exp_to_formula} 
  with the triangle inequality, we can get
  \begin{align*}
    \E{\abs*{W_2\left(\frac1n\sum_{i=1}^n\delta_{X_i}, \m\right) 
    - \sqrt{\int \abs{\nabla f^{n,t}}^2}}} &\\
    \lesssim 
    \sqrt{\frac{\log\log(n)}n} &
    +
    \xi\E{\int_M \abs{\nabla f^{n,t}}^2\de\m}^{\frac12}\fullstop
  \end{align*}
  Thus the first part of the statement follows from the choice $\xi=\frac{1}{\log n}$, 
  that gives, recalling \cref{lem:int_nabla_f}, that the leading term in the right hand side 
  is the first summand.

  The second part of the statement follows once again from \cref{eq:mu_n_to_mu_nt},
  \cref{eq:flow_to_exp} and \cref{eq:exp_to_formula}. 
  But instead of using the triangular inequality we use the elementary inequality
  \begin{equation*}
   \E{\abs{(D+C)^2-C^2}} \le 
   \E{D^2}^{1/2}\left(\E{D^2}^{1/2}+2\E{C^2}^{1/2}\right)
  \end{equation*}
  that holds for any choice of square integrable random variables $C, D$. 
  More in detail we apply the said inequality with $D=A+B$ and
  \begin{align*}
    A &= W_2(\mu^n, \m) - W_2(\mu^{n,t}, \m) \comma\\
    B &= W_2(\mu^{n,t}, \m) - \sqrt{\int_M\abs{\nabla f^{n,t}}^2\de\m} \comma\\
    C &= \sqrt{\int_M\abs{\nabla f^{n,t}}^2\de\m} \fullstop
  \end{align*}
  To estimate $\E{D^2}$ we proceed as follows
  \begin{equation*}
    \E{D^2} \lesssim \E{A^2} + \E{B^2} \le \E{W_2^2(\mu^n, \mu^{n,t})} + 
    \E{W_2^2(\mu^{n,t}, \exp(\nabla f^{n,t})_\#\m)} \comma
  \end{equation*}
  where we have applied \cref{eq:exp_to_formula}. Then \cref{eq:mu_n_to_mu_nt} and \cref{eq:flow_to_exp} provide the inequalities necessary to conclude.
\end{proof}

\begin{remark}
  In the work \cite{CaraccioloEtAl2014}, the authors claim that the higher order error term 
  should be $\bigo(\frac1n)$.
  Unfortunately with our approach it is impossible to improve the estimate on the error term 
  from $\sqrt{\log(n)\log\log(n)}/n$ to $1/n$. Indeed, even ignoring all the complex approximations and 
  estimates, our expansion involves the term $\frac{\abs{\log t}}{4\pi n}$. 
  Thence we would be obliged to set $t = \bigo\left(\frac 1n\right)$. 
  The issue is that this growth of $t$ does not allow us to exploit \cref{thm:prob_flat_density}. 
  Indeed, if $t = \bigo\left(\frac 1n\right)$, we are not able anymore to prove that 
  $A^{n,t}_\xi$ is a very likely event (even when $\xi$ is fixed) and our strategy fails.
\end{remark}

\section{Bipartite Matching}\label{sec:bipartite_matching}
Exactly as we have computed the expected cost for the semi-discrete matching problem, we are going to do
the same for the bipartite (or purely discrete) matching problem (i.e. we have to match two families 
of $n$ random points trying to minimize the sum of the distances squared). 

The approach is almost identical to the one described at the beginning of \cref{sec:semidiscrete_matching}.
Let us remark that this result is new for a general $2$-dimensional closed manifold $M$. 
Indeed in the work \cite{Ambrosio-Stra-Trevisan2018} the authors manage to handle the bipartite 
matching only when $M$ is the $2$-dimensional torus (or the square). 
Their approach is custom-tailored for the torus and thence very hard to generalize to other manifolds.

\begin{reptheorem}{thm:main_theorem_bipartite}[Main Theorem for bipartite matching]
  Let $(M,\metric)$ be a $2$-dimensional compact closed manifold (or the square $\cc01^2$) whose volume 
  measure $\m$ is a probability. 
  Let $(X_i)_{i\in\N}$ and $(Y_i)_{i\in\N}$ be two families of independent random points 
  $\m$-uniformly distributed on $M$.
  It holds the asymptotic behaviour
  \begin{equation*}
    \lim_{n\to\infty} \frac{n}{\log(n)}\cdot
    \E{W^2_2\left(\frac1n\sum_{i=1}^n\delta_{X_i}, \frac1n\sum_{i=1}^n\delta_{Y _i}\right)}
    =\frac1{2\pi} \fullstop
  \end{equation*}
\end{reptheorem}
\begin{proof}
  Let us warn the reader that in this proof the definitions of $f^{n,t}$ and of $A^{n,t}$ change. 
  The change is a natural consequence of the presence of two families of points. We decided to keep the same
  notation as the \emph{role and meaning} of the objects do not change at all.
  We will skip the parts of the proof identical to the proof of 
  \cref{thm:main_theorem_semidiscrete}.

  Let us fix a parameter $\xi=\frac1{\log n}$ and the time variable 
  $t=\log(a)^{-1}\gamma\frac{\log^3 n}{n}$ where $\gamma>0$ is a sufficiently large constant. 

  Analogously to what we have done in the semi-discrete case, let us define
  \begin{equation*}
    \mu_0^n \defeq \frac1n\sum_{i=1}^n \delta_{X_i} \comma \quad\quad
    \mu_1^n \defeq \frac1n\sum_{i=1}^n \delta_{Y_i}
  \end{equation*}
  and the associated regularized measures and densities
  \begin{equation*}
    \mu_0^{n,t} \defeq P_t^*\mu_0^n = u_0^{n,t}\m \comma \quad\quad
    \mu_1^{n,t} \defeq P_t^*\mu_1^n = u_1^{n,t}\m \fullstop
  \end{equation*}

  Let us denote with $f^{n,t}:M\to\R$ the unique function with null mean value such that
  $-\lapl f^{n,t} = u_1^{n,t}-u_0^{n,t}$. Of course it holds $f^{n,t} = f_1^{n,t}-f_0^{n,t}$ where the 
  functions $f_0^{n,t}$ and $f_1^{n,t}$ are defined exactly as in \cref{eq:definition_f} but using $\mu_0^{n,t}$ and $\mu_1^{n,t}$
  in place of $\mu$.

  Hence, we can apply \cref{thm:prob_flat_density} on $f_0^{n,t}$ and $f_1^{n,t}$ to obtain the estimate
  \begin{equation*}
    \P{(A^{n,t}_\xi)^{\complement}} \lesssim \frac1{\xi^2t^3}a^{-nt\xi^2} \fullstop
  \end{equation*}
  Here $A^{n,t}_\xi$ is defined as the intersection of the events $A^{n,t}_{\xi,\iota}$ for $\iota=0, 1$, where $A^{n,t}_{\xi,\iota}$ is
  \begin{equation*}
    A^{n,t}_{\xi,\iota} \defeq \{\norm{\nabla^2f_\iota^{n,t}}_\infty < \xi\} \fullstop
  \end{equation*}

  From now on the proof goes along the exact same lines of the proof of \cref{thm:main_theorem_semidiscrete} 
  (just replacing $\mu^n$ with $\mu_1^n$ and $\m$ with $\mu_0^n$) apart from the computation of
  \begin{equation*}
    \E{\int_M \abs{\nabla f^{n,t}}^2\de\mu_0^{n,t} } \fullstop
  \end{equation*}
  Indeed in the semi-discrete case we could blindly apply \cref{lem:int_nabla_f}, whereas now we have to
  compute it. 

  Thanks to \cref{thm:prob_flat_density} and \cref{rem:easy_restriction_rule}, we can assume 
  to be in the event $\{\norm{u^{n,t}-1}_\infty < \xi\}$ as it is so likely that its complement 
  can be safely ignored.
  Thus, if we replace $\mu_0^{n,t}$ with $\m$, we obtain
  \begin{equation}\label{eq:replace_mu}
    \E{\int_M \abs{\nabla f^{n,t}}^2\de\mu_0^{n,t} } 
    = (1+\bigo(\xi))\E{\int_M \abs{\nabla f^{n,t}}^2\de\m} \fullstop
  \end{equation}
  As already done in the proof of \cref{lem:int_nabla_f}, using the linearity of the expected value and the 
  independence of the random points we can easily compute
  \begin{equation*}
    \E{\int_M \abs{\nabla f^{n,t}}^2\de\m} 
    = \frac2n\int_M\int_M\abs{\nabla_y q_t}^2\de\m(y)\de\m(x)
  \end{equation*}
  and therefore, applying \cref{prop:on_diag_int}, we have shown
  \begin{equation*}
    \E{\int_M\abs{\nabla f^{n,t}}^2\de\m} 
    = \frac{\abs{\log(t)}}{2\pi n} + \bigo\left(\frac{1}{n}\right)
  \end{equation*}
  that, together with \cref{eq:replace_mu}, is exactly the result we needed to complete the proof.
\end{proof}

\begin{remark}[Interpolation between semi-discrete and bipartite]
  The proof given for the bipartite case is flexible enough to handle also families of random points with 
  different cardinalities. 
  Given a positive rational number $q\in\Q$ and a natural number $n\in\N$ such that $qn\in\N$, let
  $(X_i)_{1\le i\le n}$ and $(Y_i)_{1\le i\le qn}$ be two families of independent random points $\m$-uniformly
  distributed on $M$. Then, exactly as we have proven \cref{thm:main_theorem_bipartite}, we can show
  \begin{equation*}
    \E{W^2_2\left(\frac1n\sum_{i=1}^n\delta_{X_i}, \frac1{qn}\sum_{i=1}^{qn}\delta_{Y_i}\right)} 
    \sim \frac{\log(n)}{4\pi n}\left(1+\frac1q\right) \fullstop
  \end{equation*}
  Let us remark that when $q\gg 1$ we recover the result of the semi-discrete case.
\end{remark}

\appendix
\section{Stability of Vector Fields Flows}\label{app:stability_flow}
In this appendix we are going to obtain a stability result for flows of vector fields on a compact Riemannian
manifold.
This kind of results are well known, but we could not find a statement in literature that could fit exactly our
needs. The proof has a very classical flavor, borrowing the majority of the ideas from the 
uniqueness theory for ordinary differential equations.
Nonetheless it might seem a little technical as we are working on a Riemannian manifold and we are using 
both flows of
vector fields and the exponential map.

Given a compact closed Riemannian manifold $M$, we assume in this section that $X\in\chi(M)$ is a vector field such 
that $\norm{\nabla X}_\infty<\frac12$ (see \cref{def:tensor_infinity_norm}) and such that
$\norm{X}_\infty$ is sufficiently small with respect to $\injradius(M)$.

\begin{proposition}\label{prop:stability_flow}
  Under the previous assumptions on $X$,
  if $(Y_t)_{0\le t\le 1}$ is a time dependent vector field such that, for a 
  suitable $0<\xi<1$, it holds pointwise
  \begin{equation*}
    \abs{Y_t-X} < \xi \abs{X}\comma
  \end{equation*}
  then 
  \begin{equation*}
    d(\exp_p(X), F_1^{Y_t}(p)) \lesssim (\norm{\nabla X}_\infty + \xi)\abs{X}(p)
  \end{equation*}
  for any $p\in M$.
  In particular, for any $\mu\in\prob(M)$, it holds
  \begin{equation*}
    W_2^2(\exp(X)_\#\mu, (F_1^{Y_t})_\#\mu) 
    \lesssim (\norm{\nabla X}_\infty+\xi)^2\int_M \abs{X}^2\de\mu \fullstop
  \end{equation*}
\end{proposition}
\begin{proof}
  Let us begin with a technical lemma that proves the result for a single pathline of the flow 
  (denoted by $\gamma$) if, instead of the exponential map, the flow of $X$ is considered.
  \begin{lemma}\label{lem:stability_curve}
    Under the previous assumptions on $X$,
    if $\gamma:\cc01\to M$ is a $C^1$ curve such that $\abs{\gamma'-X\circ\gamma}\le\xi\abs{X}\circ\gamma$, then it holds
    \begin{equation*}
      d\left(F_1^X(\gamma(0)), \gamma(1)\right)\le 4\xi\abs{X}(\gamma(0)) \comma
    \end{equation*}
    where $F_t^X:M\to M$ is the flow induced by $X$ at time $t$.
  \end{lemma}
  \begin{proof}
    The approach we are going to carry on is standard and it involves finding a suitable 
    differential inequality for the left hand side that automatically implies the desired inequality.

    Let us define $p\defeq \gamma(0)$.
    \newcommand{\D}{\mathscr{D}}%
    As $X$ is sufficiently small, we can be sure that, at any time $0\le t\le 1$, $F_t^X(p)$ 
    and $\gamma(t)$ are very near. Therefore the function $\D:\cc01\to\R$ defined as 
    $\D(t)\defeq d(F_t^X(p), \gamma(t))$ is Lipschitz and differentiable where it does not vanish.

  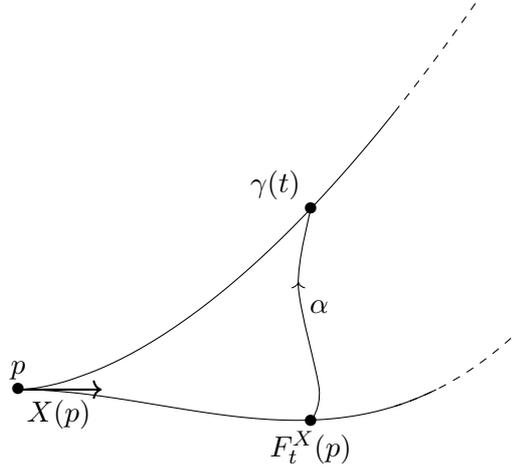
\begin{figure}[htb]
    \centering
    \begin{tikzpicture}[scale=5.5]
\coordinate (P) at (0,0);
\coordinate (A) at (0.7,-0.076);
\coordinate (B) at (0.7, 0.435);

\draw[domain=0:0.9,smooth,variable=\x] plot ({\x},{0.8*\x^1.7});
\draw[domain=0.9:1.1,smooth,variable=\x,dashed] plot ({\x},{0.8*\x^1.7});
\draw[domain=0:1,smooth,variable=\x] plot ({\x},{-0.5*\x^2 + 0.5*\x^3});
\draw[domain=0.9:1.2,smooth,variable=\x, dashed] plot ({\x},{-0.5*\x^2 + 0.5*\x^3});
\draw plot [smooth] coordinates {(A) (0.72, 0) (0.67, 0.25) (B)};%alpha
% ‎\draw ‎[‎smooth,‎samples=100‎,domain=0:1‎‎] ‎plot({\x},‎‎‎‎‎‎‎‎{\x*\x‎‎‎‎‎‎‎‎‎‎‎‎‎‎});‎‎

\node at (P) {\textbullet};
\node[above] at (P) {$p$};
\node at (A) {\textbullet};
\node[below] at (A) {$F^X_t(p)$};
\node at (B) {\textbullet};
\node[above left] at (B) {$\gamma(t)$};
\node at (0.72, 0.2) {$\alpha$};
\draw[->,thick] (0,0)--(0.2,0) node[below left]{$X(p)$};
\draw[->] (0.67, 0.25) -- (0.67, 0.26);

\end{tikzpicture}
    \caption{Curves involved in the proof of \cref{lem:stability_curve}.}
  \end{figure}
    Let us fix a time $0<t\le 1$ such that $\D(t)>0$ and let us denote with $\alpha:\cc{0}{\D(t)}\to M$ 
    a minimizing unit speed geodesic from $F_t^X(p)$ to $\gamma(t)$.
    Hence, considering the expression of the differential of the Riemannian distance, we have
    \begin{equation*}
      \frac{\de}{\de t}\D(t) 
      = \scalprod{\alpha'(\D(t))}{\gamma'(t)}-\scalprod{\alpha'(0)}{X(F_t^X(p))}\fullstop
    \end{equation*}
    Denoting $\Xpar:\alpha\to T\alpha$ the parallel transportation of $X$ from $\alpha(0)=F_t^X(p)$ to 
    $\alpha(\D(t))=\gamma(t)$, we can carry on our computations
    \begin{align*}
      &=\scalprod{\alpha'(\D(t))}{\gamma'(t)-\Xpar(\gamma(t))}
      \le \abs{\gamma'(t)-\Xpar(\gamma(t))}
      \le \abs{\gamma'(t)-X(\gamma(t))} + \abs{X(\gamma(t))-\Xpar(\gamma(t))}\\
      &\le \xi\abs{X}(\gamma(t))+\D(t)\norm{\nabla X}_{\infty}
      \le \xi\left(\abs{X}(F_t^X(p))+\D(t)\norm{\nabla X}_{\infty}\right) 
      +\D(t)\norm{\nabla X}_{\infty}\\
      &= \xi\abs{X}(F_t^X(p))+(1+\xi)\D(t)\norm{\nabla X}_{\infty}\\
      &\le \xi\left(\abs{X}(p) + d(p,F_t^X(p))\norm{\nabla X}_{\infty}\right)+
      (1+\xi)\D(t)\norm{\nabla X}_{\infty}\fullstop
    \end{align*}
    Our goal is now to estimate $d(p,F_t^X(p))$ with $\abs{X}(p)$. 
    That is a fairly easy task and can be achieved differentiating the said quantity. Indeed it holds
    \begin{equation*}
      \frac{\de}{\de t} d(p, F_t^X(p)) \le \abs{X(F_t^X(p))} 
      \le \abs{X(p)}+d(p, F_t^X(p))\norm{\nabla X}_\infty
    \end{equation*}
    and hence we can obtain the bound
    \begin{equation}\label{eq:easy_abs_bound}
      d(p,F_t^X(p))
      \le \abs{X}(p)\frac{e^{t\norm{\nabla X}_{\infty}}-1}{\norm{\nabla X}_{\infty}}\fullstop
    \end{equation}
    Replacing \cref{eq:easy_abs_bound} into our chain of inequalities we can finally get
    \begin{align*}
      \frac{\de}{\de t}\D(t)
      &\le \xi\abs{X}(p)e^{\norm{\nabla X}_{\infty}}
      +(1+\xi)\D(t)\norm{\nabla X}_{\infty} \\
      &\le 2\xi\abs{X}(p)+\D(t) \fullstop
    \end{align*}
    The statement follows easily integrating this last inequality.
  \end{proof}
  As anticipated, the main issue with the previous lemma is that instead of the exponential map, it uses
  the flow of $X$. We will use a trick that involves applying the lemma again to replace the
  flow with the exponential map. In order for our trick to work we need the following simple estimate.
 
  \begin{lemma}\label{lem:derivative_geodesic}
    Under the previous assumptions on $X$,
    for any $0\le t\le 1$ and for any $p\in M$, it holds
    \begin{equation*}
      \abs*{\frac{\de}{\de t}\exp_p(tX)-X(\exp_p(tX))}\le 
      2\norm{\nabla X}_{\infty}\cdot\abs{X(\exp_p(tX))} \fullstop
    \end{equation*}
  \end{lemma}
  \begin{proof}
    Let $\gamma(t)\defeq \exp_p(tX)$ be the unique geodesic with initial data $\gamma(0)=p$ and 
    $\gamma'(0)=X(p)$.
    Whenever $X(\gamma(t))\not=\gamma'(t)$, it holds
    \begin{equation*}
      \frac{\de}{\de t}\abs*{\gamma'(t)-X(\gamma(t))} = 
      \frac{\scalprod{\nabla_{\gamma'}X}{X(\gamma(t))-\gamma'(t)}}{\abs*{\gamma'(t)-X(\gamma(t))}}
    \end{equation*}
    and thus
    \begin{equation*}
      \abs*{\frac{\de}{\de t}\abs*{\gamma'(t)-X(\gamma(t))}} \le \norm{\nabla X}_\infty\abs{X(p)}
    \end{equation*}
    that, recalling the assumption $\norm{\nabla X}_\infty\le\frac12$, implies the statement.
  \end{proof}

  We can now prove the first part of the proposition.
  Let us fix a point $p\in M$. Thanks to \cref{lem:stability_curve} we know
  \begin{equation*}
    d(F_1^X(p), F_1^{Y_t}(p)) \le 4\xi\abs{X}(p) \fullstop
  \end{equation*}
  Furthermore, what we have shown in \cref{lem:derivative_geodesic} is exactly what we need to apply again
  \cref{lem:stability_curve} to the geodesic $\exp_p(tX)$, obtaining
  \begin{equation*}
    d(F_1^X(p), \exp_p(X)) \le 8\norm{\nabla X}_\infty\abs{X}(p) \fullstop
  \end{equation*}
  Joining this two inequalities we obtain the first part of the statement.

  The part of the statement about the Wasserstein distance between the push-forward measures follows trivially
  from what we have already proven, thanks to the following well-known lemma.

  \begin{lemma}\label{lem:near_maps_near_measures}
    Let $(X,\m)$ be a probability space and let $(Y, d)$ be a metric space. 
    Given two maps $f, g:X\to Y$ it holds
    \begin{equation*}
      W_p(f_\#\m, g_\#\m) \le \norm{d(f, g)}_{L^p(\m)}\fullstop
    \end{equation*}
  \end{lemma}
  \begin{proof}   
    Let us consider the map $T:X\to Y\times Y$ defined as $T(x)\defeq (f(x), g(x))$.
    The key observation is that $T_\#\m$ is a transport plan from $f_\#\m$ to $g_\#\m$,  therefore
    \begin{equation*}
      W_p^p(f_\#\m, g_\#\m)\le\int_{Y\times Y} d(y_1, y_2)^p\de T_\#\m (y_1,y_2)
      = \int_X d(f(x), g(x))^p\de\m(x) \fullstop
    \end{equation*}
  \end{proof}  
\end{proof}
\begin{remark}
  If the manifold $M$ has a smooth boundary, the previous proposition can be easily adapted 
  as far as the vector fields $X$ and $Y_t$ are both tangent to the boundary. 
  Indeed, under this assumption, if we extend arbitrarily the manifold $M$ to a closed compact 
  manifold\footnote{The extension is necessary to give a sense to the exponential map.} 
  $\tilde M$, the result keeps holding exactly as it is stated.
\end{remark}

\section{Heat Kernel on the Square} \label{app:square}
We will explicitly construct the Neumann heat kernel on the domain $\cc01^2$ in order to
show the validity of \cref{thm:heat_deriv} in this setting. The expression of the Neumann heat kernel
on the square is folklore, but for the ease of the reader we report it here.

Let $G<\Isom(\R^2)$ be the subgroup of isometries generated by the four transformations
\begin{itemize}
 \item $(x_1,x_2)\mapsto (x_1, x_2+2)$,
 \item $(x_1,x_2)\mapsto (x_1+2, x_2)$,
 \item $(x_1,x_2)\mapsto (-x_1, x_2)$,
 \item $(x_1,x_2)\mapsto (x_1, -x_2)$.
\end{itemize}
The group $G$ is generated by the reflections with respect to any horizontal or vertical line with integer
coordinates.
The square $\cc01^2$ is a fundamental cell for the action of the group $G$.
For any $x\in\R^2$, we will denote $Gx$ the orbit of the point $x$ under the action of $G$.
\begin{figure}[htb]
  \centering
  \begin{tikzpicture}[scale=1.5]

% Coordinates of x and Gx
\pgfmathsetmacro{\X}{0.1}
\pgfmathsetmacro{\Y}{0.7}

\coordinate (X) at (\X, \Y);
\coordinate (X1) at (2-\X, \Y);
\coordinate (X2) at (-\X, \Y);
\coordinate (X3) at (\X, 2-\Y);
\coordinate (X4) at (\X, -\Y);
\coordinate (X5) at (2-\X, 2-\Y);
\coordinate (X6) at (2-\X, -\Y);
\coordinate (X7) at (-\X, 2-\Y);
\coordinate (X8) at (-\X, -\Y);
\coordinate (X9) at (\X+2, 2-\Y);
\coordinate (X10) at (\X+2, \Y);
\coordinate (X11) at (\X+2, -\Y);
\coordinate (X12) at (2-\X, \Y-2);
\coordinate (X13) at (\X, \Y-2);
\coordinate (X14) at (-\X, \Y-2);
\coordinate (X15) at (\X+2, \Y-2);
\coordinate (X16) at (\X-2, 2-\Y);
\coordinate (X17) at (\X-2, \Y);
\coordinate (X18) at (\X-2, -\Y);
\coordinate (X19) at (\X-2, \Y-2);

% The description of Gx
\pgfmathsetmacro{\PX}{2}
\pgfmathsetmacro{\PY}{2.7}

\draw [fill=black] (\PX, \PY) circle [radius=1pt]; %node[above]{$x'$};
\draw [fill=black] (\PX + 0.13, \PY + 0.1) circle [radius=1pt]; %node[above]{$x'$};
\draw [fill=black] (\PX + 0.13, \PY - 0.05) circle [radius=1pt]; %node[above]{$x'$};
\node at (\PX-0.09, \PY +0.02) {$\{$};
\node at (\PX+0.58, \PY +0.02) {$\}=Gx$};

% Vertical axis
\draw[-] (0, -1.5)--(0, 2.5);
\draw[-] (1, -1.5)--(1, 2.5);
\draw[-] (-1, -1.5)--(-1, 2.5);
\draw[-] (2, -1.5)--(2, 2.5);
\draw[-, dashed] (-2, -1.5)--(-2, 2.5);
\draw[-, dashed] (3, -1.5)--(3, 2.5);
% Horizontal axis
\draw[-] (-2.3, 0)--(3.3, 0);
\draw[-] (-2.3, 1)--(3.3, 1);
\draw[-, dashed] (-2.3, -1)--(3.3, -1);
\draw[-, dashed] (-2.3, 2)--(3.3, 2);

% Drawing x
\fill [white!90!black] (0.01, 0.01) rectangle (1-0.01,1-0.01);

% Drawing Gx
\draw [fill=white, thick] (X) circle [radius=1pt] node[below right]{$x$};
\draw [fill=black] (X1) circle [radius=1pt]; %node[above]{$x'$};
\draw [fill=black] (X2) circle [radius=1pt]; %node[above]{$x'$};
\draw [fill=black] (X3) circle [radius=1pt]; %node[above]{$x'$};
\draw [fill=black] (X4) circle [radius=1pt]; %node[above]{$x'$};
\draw [fill=black] (X5) circle [radius=1pt]; %node[above]{$x'$};
\draw [fill=black] (X6) circle [radius=1pt]; %node[above]{$x'$};
\draw [fill=black] (X7) circle [radius=1pt]; %node[above]{$x'$};
\draw [fill=black] (X8) circle [radius=1pt]; 
\draw [fill=black] (X9) circle [radius=1pt]; 
\draw [fill=black] (X10) circle [radius=1pt]; 
\draw [fill=black] (X11) circle [radius=1pt]; 
\draw [fill=black] (X12) circle [radius=1pt]; 
\draw [fill=black] (X13) circle [radius=1pt]; 
\draw [fill=black] (X14) circle [radius=1pt]; 
\draw [fill=black] (X15) circle [radius=1pt]; 
\draw [fill=black] (X16) circle [radius=1pt]; 
\draw [fill=black] (X17) circle [radius=1pt]; 
\draw [fill=black] (X18) circle [radius=1pt]; 
\draw [fill=black] (X19) circle [radius=1pt]; 

\end{tikzpicture}
  \caption{Chosen $x\in\cc01^2$, the points in $Gx$ are generated by reflections.}
\end{figure}
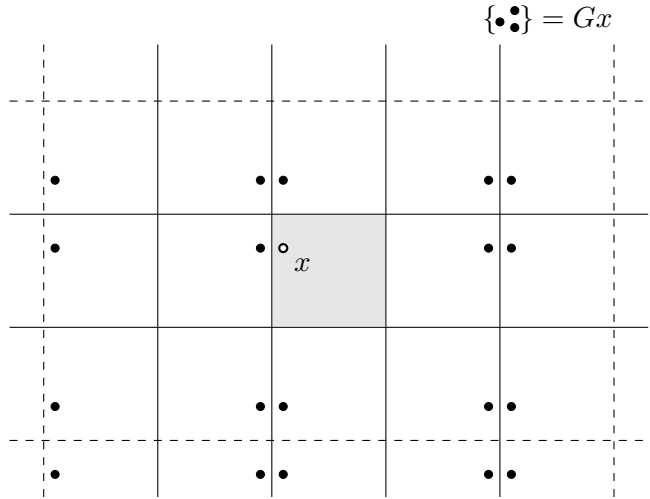

Let $p_t:\R^2\times\R^2\to\oo0\infty$ be the heat kernel on the plane, that is
\begin{equation*}
  p_t(x, y) = \frac1{4\pi t}e^{\frac{-\abs{x-y}^2}{4t}} \fullstop
\end{equation*}

Let us define the associated kernel $\tilde p_t:\R^2\times\R^2\to\oo0\infty$ as
\begin{equation*}
  \tilde p_t(x, y) \defeq \sum_{x'\in Gx} p_t(x', y) \fullstop
\end{equation*}

Our goal is showing that $\tilde p_t$ is exactly the (Neumann) heat kernel for the domain $\cc01^2$.
All the needed verifications are readily done, apart from the fact that $\tilde p_t$ 
satisfies the Neumann boundary conditions. 
This property follows from the following simple symmetries of $\tilde p_t$:
\begin{equation*}
  \forall x, y\in\R^2, \forall g\in G: \quad 
  \tilde p_t(x, y) = \tilde p_t(g(x), y) = \tilde p_t(x, g(y)) \fullstop
\end{equation*}

It is now time to prove \cref{thm:heat_deriv} for $\tilde p_t$.
\begin{proof}[Proof of \cref{thm:heat_deriv} for $M=\cc01^2$]
  First, we are going to prove explicitly \cref{thm:heat_deriv} for the heat kernel on the plane.
  
  As one can show by induction, it holds
  \begin{equation*}
      \partial^{n_1}_{y_1}\partial_{y_2}^{n_2} p_t(x,y) = p_t(x,y) \cdot
      \sum_{\substack{2\beta-\alpha=n_1+n_2 \\ 0\le \alpha,\beta\le n_1+n_2}}\sum_{m_1+m_2 = \alpha} c_{n_1,n_2}^{m_1,m_2}(x_1-y_1)^{m_1}(x_2-y_2)^{m_2}t^{-\beta} \comma
  \end{equation*}
  for some suitable coefficients $c_{n_1,n_2}^{m_1,m_2}$. From this formula, it follows
  \begin{gather}\begin{split}\label{eq:pt_computations}
    \abs{\nabla^N p_t(x, y)} 
    &\lesssim_N p_t(x, y)\sum_{\substack{2\beta-\alpha=N \\ 0\le \alpha,\beta\le N}} 
    \abs{x-y}^\alpha t^{-\beta} \\
    &\lesssim_N t^{-1-N/2}\sum_{0\le\alpha\le N} 
    \left(\frac{\abs{x-y}^2}{4t}\right)^{\alpha/2}e^{-\frac{\abs{x-y}^2}{4t}}\fullstop
  \end{split}\end{gather}
  
  For any $0<w\le u$ and 
  $0\le\alpha\le N$ it holds
  \begin{equation}\label{eq:elementary_ineq}
    u^{\alpha/2}e^{-u}\lesssim_N (1+w^{N/2})e^{-w} \fullstop
  \end{equation}
  If we apply the latter inequality (for the special case $u=w$) in \cref{eq:pt_computations}, we get
  \begin{equation}\label{eq:pt_deriv}
    \abs{\nabla^N p_t(x, y)} 
    \lesssim_N t^{-1-N/2}\left(1 
	    + \left(\frac{\abs{x-y}^2}{4t}\right)^\frac N2\right) e^{-\frac{\abs{x-y}^2}{4t}} \fullstop
  \end{equation}
 
  Let us move our attention to $\tilde p_t$.
  If we take $x,y\in\cc01^2$ and $t\le 1$, applying \cref{eq:pt_deriv} we are able to show
  \begin{equation*}
    \abs{\nabla^N \tilde p_t(x, y)} \le \sum_{x'\in Gx} \abs{\nabla^N p_t(x', y)}
      \lesssim_N  t^{-1-N/2}\sum_{x'\in Gx} \left(1+ \left(\frac{\abs{x'-y}^2}{4t}\right)^{\frac N2}\right) 
      e^{-\frac{\abs{x'-y}^2}{4t}} \fullstop
  \end{equation*}
  Using that $\abs{x-y}\le \abs{x'-y}$ for any $x'\in Gx$, recalling \cref{eq:elementary_ineq} and
  noticing the exponential decay of the quantities when $x'$ goes to infinity, from the latter inequality 
  we can deduce
  \begin{align*}
    \abs{\nabla^N \tilde p_t(x, y)} 
    &\lesssim_N  t^{-1-N/2}\left(1+ \left(\frac{\abs{x-y}^2}{4t}\right)^{\frac N2}\right) 
		e^{-\frac{\abs{x-y}^2}{4t}} \\
    &= t^{-1}e^{-\frac{\abs{x-y}^2}{4t}}
		\left(t^{-N/2}+\left(\frac{\abs{x-y}}{2t}\right)^N\right)
    \lesssim \tilde p_t(x, y)
	     \left(t^{-N/2}+\left(\frac{\abs{x-y}}t\right)^N\right)
  \end{align*}
  that is the desired result.
\end{proof}

\printbibliography

\end{document}